\documentclass[a4paper]{amsproc}
\usepackage{amsmath,enumitem}
\usepackage{amssymb, amsthm, mathrsfs}
\usepackage{comment}

\usepackage{hyperref}
\usepackage{here}

\usepackage[top=30truemm, bottom=25truemm, left=25truemm, right=25truemm]{geometry}

\usepackage[dvipsnames]{xcolor}
\usepackage[capitalize,nameinlink]{cleveref}
\hypersetup{
	colorlinks=true,       % false: boxed links; true: colored links
	linkcolor=brown,          % color of internal links
	citecolor=brown,        % color of links to bibliography
}

\newtheorem{theoremcounter}{Theorem Counter}[section]

\theoremstyle{definition}
\newtheorem{definition}[theoremcounter]{Definition}
\newtheorem{remark}[theoremcounter]{Remark}
\newtheorem{example}[theoremcounter]{Example}

\theoremstyle{plain}
\newtheorem{lemma}[theoremcounter]{Lemma}
\newtheorem{proposition}[theoremcounter]{Proposition}
\newtheorem{corollary}[theoremcounter]{Corollary}
\newtheorem{conjecture}[theoremcounter]{Conjecture}
\newtheorem{theorem}[theoremcounter]{Theorem}

\numberwithin{equation}{section}

\newcommand{\Z}{{\mathbb Z}}
\newcommand{\Q}{{\mathbb Q}}
\def\sts#1#2{\genfrac{\{}{\}}{0pt}{}{#1}{#2}}

\title[Stephan's observation]{Remarkable relations between the central binomial series, Eulerian polynomials, and poly-Bernoulli numbers}

\author{Be\'ata B\'enyi}
\address{\noindent Faculty of Water Sciences, University of Public Service, Baja, HUNGARY}
\email{benyi.beata@uni-nke.hu}

\author{Toshiki Matsusaka}
\address{Institute for Advanced Research, Nagoya University, Furo-cho, Chikusa-ku, Nagoya, 464-8601, JAPAN}
\email{matsusaka.toshiki@math.nagoya-u.ac.jp}

\date{\today}
\thanks{The second author was supported by JSPS KAKENHI Grant Number 20K14292 and 21K18141.}

\begin{document}

\begin{abstract}
	The central binomial series at negative integers are expressed as a linear combination of values of certain two polynomials. We show that one of the polynomials is a special value of the bivariate Eulerian polynomial and the other polynomial is related to the antidiagonal sum of poly-Bernoulli numbers. As an application, we prove Stephan's observation from 2004.
\end{abstract}

\maketitle

%-------------------------------------------------------------------
\section{Introduction}
%-------------------------------------------------------------------

The \emph{central binomial series} is a Dirichlet series defined by
\begin{align}\label{dir_ser}
	\zeta_{CB}(s) = \sum_{n=1}^\infty \frac{1}{n^s {2n \choose n}} \quad (s \in \mathbb{C}).
\end{align}
Borwein, Broadhurst, and Kamnitzer~\cite{BorweinBroadhurstKamnitzer2001} studied 
special values $\zeta_{CB}(k)$ at positive integers and recovered some remarkable connections. 
A classical evaluation is $\zeta_{CB}(4) = \frac{17\pi^{4}}{3240} = \frac{17}{36}\zeta(4)$.
In particular, for $k\geq 2$ Borwein--Broadhurst--Kamnitzer showed that $\zeta_{CB}(k)$ can be written as a $\mathbb{Q}$-linear combination of multiple zeta values and multiple Clausen and Glaisher values. 

On the other hand, Lehmer~\cite{Lehmer1985} proved that for $k\leq 1$, $\zeta_{CB}(k)$ is a $\mathbb{Q}$-linear combination of $1$ and $\pi/\sqrt{3}$. For example, we have
\begin{align*}
	\zeta_{CB}(1) = \frac{1}{3} \frac{\pi}{\sqrt{3}}, \quad \zeta_{CB}(0) = \frac{1}{3} + \frac{2}{9} \frac{\pi}{\sqrt{3}}, \quad \zeta_{CB}(-1) = \frac{2}{3} + \frac{2}{9} \frac{\pi}{\sqrt{3}}, \quad \zeta_{CB}(-2) = \frac{4}{3} + \frac{10}{27} \frac{\pi}{\sqrt{3}}.
\end{align*}
He considered the general sum
\begin{align*}
	\sum_{n=1}^{\infty}\frac{(2x)^{2n}}{n\binom{2n}{n}} = \frac{2x\arcsin (x)}{\sqrt{1-x^2}} \quad (|x|<1)
\end{align*}
and its derivatives to derive interesting series evaluations. 
More precisely, Lehmer provided the following explicit formula for the special values $\zeta_{CB}(k)$ at negative integers. 
Define two sequences of polynomials $(p_k(x))_{k \geq -1}$ and $(q_k(x))_{k \geq -1}$ by the initial values $p_{-1}(x) = 0, q_{-1}(x) = 1$ and the recursion
\begin{align}\label{p-q-def}
	\begin{split}
		p_{k+1}(x) &= 2(kx+1) p_k(x) + 2x(1-x) p'_k(x) + q_k(x),\\
		q_{k+1}(x) &= \left(2(k+1)x+1 \right) q_k(x) + 2x(1-x) q'_k(x).
	\end{split}
\end{align}
Then for $k \geq -1$, we have
\begin{align}\label{Lehmer-central-binomial}
	\sum_{n=1}^{\infty} \frac{(2n)^k(2x)^{2n}}{\binom{2n}{n}} = \frac{x}{(1-x^2)^{k+\frac{3}{2}}}\left(x\sqrt{1-x^2}p_k(x^2)+\arcsin(x) q_k(x^2)\right).
\end{align}
Consequently, 
\begin{align}\label{ZetaBC}
	\zeta_{CB}(-k) = \frac{1}{3} \left(\frac{2}{3} \right)^k p_k \left(\frac{1}{4} \right) + \frac{1}{3} \left(\frac{2}{3} \right)^{k+1} q_k \left(\frac{1}{4} \right) \frac{\pi}{\sqrt{3}} \in \Q + \Q \frac{\pi}{\sqrt{3}}.
\end{align}
The first few polynomials are: $p_0(x) =1$, $p_1(x) = 3$, $p_2(x) = 8x+7$; and $q_0(x) = 1$, $q_1(x) =2x+1$, $q_2(x) = 4x^2+10x+1$.

In 2004, Stephan~\cite[A098830]{OEIS} observed that the rational part of \eqref{ZetaBC} is nothing but (a third of) a sum of poly-Bernoulli numbers of negative indices. 
Poly-Bernoulli numbers $B_n^{(k)}$ are a generalization of classical Bernoulli numbers using polylogarithm functions and were introduced by Kaneko~\cite{Kaneko1997}. We will give a precise definition in \cref{s3}.
\begin{conjecture}\cite[stated by Kaneko, Stephan's conjecture]{Kaneko2008}\label{Stephan's conjecture}
	For any $n \geq 0$,
	\[
	\left(\frac{2}{3} \right)^n p_n \left(\frac{1}{4} \right) = \sum_{k=0}^n B_{n-k}^{(-k)}.
	\]
\end{conjecture}

In this article, we connect both polynomials $p_n(x)$ and $q_n(x)$ to known numbers and polynomials. 
More precisely, we prove Stephan's conjecture (relating this way $p_n(x)$ to the poly-Bernoulli numbers) using the fact that the polynomial sequence $q_n(x)$ is a generalization of the classical Eulerian polynomials.

%----------------------------------------------------------------
\section{The polynomials $q_n(x)$ and bivariate Eulerian polynomials}
%----------------------------------------------------------------

Eulerian polynomials have been studied by Euler himself. Since then they have been studied and became classical. Several extensions, generalizations and applications are known today. 

Let $\mathfrak{S}_n$ denote the set of permutations $\pi = \pi_1 \pi_2 \dots \pi_n$ of $[n] = \{1, 2, \dots, n\}$. For each $\pi\in \mathfrak{S}_n$, the excedance set is defined as $\mathrm{Exc}(\pi) = \{i \in [n] \mid \pi_i > i\}$. We set $\mathrm{exc}(\pi) = |\mathrm{Exc}(\pi)|$. It is well-known that the \emph{Eulerian number} $A(n,k)$ counts the number of permutations $\pi \in \mathfrak{S}_n$ with $\mathrm{exc}(\pi) = k$. For instance, $A(3,1) = 4$ because there are $4$ permutations of $\{1,2,3\}$ with $\mathrm{exc}(\pi) = 1$, namely, $132, 213, 312, 321$. 
A map $f: \mathfrak{S}_n \to \Z_{\geq 0}$ satisfying $|\{\pi \in \mathfrak{S}_n \mid f(\pi) = k\}| = A(n,k)$ is often called an \emph{Eulerian statistic}. The map $\mathrm{exc}$ is an example of Eulerian statistics. By Foata's fundamental transformation, it is also known that the number of permutations with $k$ excedances is the same as the number of permutations with $k$ descents, or equivalently formulated, with $k+1$ ascending runs, (see B\'{o}na's book~\cite{Bona2012}). 

The \emph{Eulerian polynomial} is defined  by
\[
	A_n(x) = \sum_{\pi \in \mathfrak{S}_n} x^{\mathrm{exc}(\pi)} = \sum_{k=0}^{n-1} A(n,k) x^k.
\]
The generating function of the Eulerian polynomials is given as
\begin{align*}
	\sum_{n=0}^{\infty} A_n(x)\frac{t^n}{n!} = \frac{1-x}{e^{t(x-1)}-x}.
\end{align*}

For a more detailed history and properties on the Eulerian numbers (polynomials) and Eulerian statistics, the articles~\cite{Bona2012, FoataSchutzenberger1970, Petersen2015} are good references.

We recall now a generalization of the Eulerian polynomial introduced by Foata--Sch\"{u}tzenberger~\cite[Chapter IV-3]{FoataSchutzenberger1970}.
Here we define a shifted version. Let $\mathrm{cyc}(\pi)$ denote the number of cycles in the disjoint cycle representation of $\pi \in \mathfrak{S}_n$.

\begin{definition}[Bivariate Eulerian polynomial]
	For any integer $n \geq 0$, let $F_0(x,y) = 1$ and define
	\[
	F_n(x,y) = \sum_{\pi \in \mathfrak{S}_n} x^{\mathrm{exc}(\pi)} y^{\mathrm{cyc}(\pi)}, \qquad (n > 0).
	\]
\end{definition}

\begin{example}
	We have $F_3(x,y) = y^3 + 3xy^2 + x^2y + xy$.
	\begin{table}[H]
  \centering
  \begin{tabular}{|c|cccccc|}
    \hline
    $\mathfrak{S}_3$ & $123 = (1)(2)(3)$ & $132 = (1)(23)$ & $213 = (12)(3)$ & $231 = (123)$ & $312 = (132)$ & $321 = (13)(2)$ \\
    \hline
    $\mathrm{exc}(\pi)$ & $0$ & $1$ & $1$ & $2$ & $1$ & $1$ \\
    $\mathrm{cyc}(\pi)$ & $3$ & $2$ & $2$ & $1$ & $1$ & $2$ \\
    \hline
  \end{tabular}
\caption{Permutations in $\mathfrak{S}_3$ with their weights.}
\end{table}
\end{example}

The generating function of the bivariate Eulerian polynomials is given by
\begin{align}\label{generating-Eulerian}
	\mathscr{F}(x,y; t) := \sum_{n=0}^{\infty} F_n(x,y)\frac{t^n}{n!} = \left(\frac{1-x}{e^{t(x-1)}-x}\right)^y.
\end{align}

Savage--Viswanathan~\cite{SavageViswanathan2012} derived several identities for the polynomials. Here we recall their recursion formula.

\begin{proposition}
	For $n \geq 0$,
	\[
	F_{n+1}(x,y) = \left(x(1-x) \frac{d}{dx} + nx + y \right) F_n(x,y),
	\]
	with the initial value $F_0(x,y)=1$. 
\end{proposition}

Note that by the definition, we have $F_n(x,1) = A_n(x)$. Moreover, for $y = r \in \Z_{\geq 2}$, the polynomials $F_n(x,r)$ are the $r$-Eulerian polynomials originally studied by Riordan~\cite{Riordan1958}. In addition, we have $F_{n+1}(x,-1) = -(x-1)^n$ and $F_{n+1}(1,y)= y(y+1)\cdots(y+n)$ for any $n \geq 0$. 

 The surprising fact is however that the values at $y=1/k$, for any positive integer $k$, have also nice combinatorial interpretations. 
Namely, for a sequence $\mathbf{s}=(s_i)_{i\geq 1}$ of positive integers, let the \emph{$\mathbf{s}$-inversion sequence} of length $n$ be defined as 
\begin{align*}
	I_n^{(\mathbf{s})} = \{(e_1,\ldots, e_n)\in \mathbb{Z}^n \mid 0\leq e_i<s_i \text{ for } 1\leq i\leq n\}.
\end{align*}
The ascent statistic on $e\in I_n^{(\mathbf{s})}$ is 
\begin{align*}
	\text{asc}(e) = \left|\left\{0\leq i<n:\frac{e_i}{s_i}<\frac{e_{i+1}}{s_{i+1}}\right\}\right|,
\end{align*}
with the convention that $e_0/s_0 = 0$.
Then the \emph{$\mathbf{s}$-Eulerian polynomials} are defined by
\begin{align*}
	E_n^{(\mathbf{s})}(x) = \sum_{e\in I_n^{(\mathbf{s})}}x^{\text{asc}(e)}.
\end{align*}
For more properties of the $\mathbf{s}$-Eulerian polynomials, see also Savage--Visontai~\cite{SavageVisontai2015}.
Note that for 
$\mathbf{s} = (i)_{i \geq 1} = (1, 2, 3, \dots)$, the $\mathbf{s}$-Eulerian polynomials are the classical Eulerian polynomials, $E_n^{(\mathbf{s})} = A_n(x)$. 
Savage--Viswanathan~\cite{SavageViswanathan2012} showed that for $\mathbf{s} = ((i-1)k+1)_{i \geq 1} = (1, k+1, 2k+1, 3k+1, \dots)$ where $k$ is a positive integer, it holds
\begin{align*}
	E_n^{(\mathbf{s})}(x) = k^nF_n \left(x,\frac{1}{k} \right).
\end{align*}
They called the coefficients in this special case the \emph{$1/k$-Eulerian numbers}. The $1/k$-Eulerian numbers play role in the theory of $k$-lecture hall polytopes~\cite{SavageViswanathan2012} and enumerate certain statistics in $k$-Stirling permutations~\cite{MAMansour2015}.
We now show that for $k=2$, the $E_n^{(1, 3, 5, \dots)}(x) = 2^n F_n(x, 1/2)$ is the same as the $q_n(x)$ polynomial sequence in Lehmer's identity.

\begin{table}[H]
	\centering
	\begin{tabular}{|c||c|c|} \hline
		$n$ & $2^n F_n(x,1/2)$ & $q_n(x)$ \\ \hline \hline
		$-1$ & $-$ & $1$ \\
		$0$ & $1$ & $1$ \\
		$1$ & $1$ & $2x+1$  \\
		$2$ & $2x+1$ & $4x^2+10x+1$ \\ 
		$3$ & $4x^2 + 10x + 1$ & $8x^3+60x^2+36x+1$ \\
		$4$ & $8x^3+60x^2+36x+1$ & $16x^4+296x^3+516x^2+116x+1$ \\
		$5$ & $16x^4 + 296x^3 + 516x^2 + 116x + 1$ & $\cdots$ \\\hline
	\end{tabular}
    \caption{The polynomials $2^n F_n(x,1/2)$ and $q_n(x)$.}
\end{table}

\begin{theorem}\label{Q-explicit}
	The generating function
	\[
	Q(x,t) := \sum_{n=0}^\infty q_{n-1}(x) \frac{t^n}{n!}
	\]
	equals $\mathscr{F}(x,1/2;2t)$, that is, $q_n(x) = 2^{n+1} F_{n+1}(x, 1/2)$ for any $n \geq -1$. 
\end{theorem}

\begin{proof}
	By translating the recursion in \eqref{p-q-def}, the generating function $Q(x,t)$ is characterized by the differential equation
	\[
	\left((2xt-1) \frac{d}{dt} + 2x(1-x) \frac{d}{dx} + 1 \right) Q(x,t) = 0
	\]
	and the initial condition $Q(x,0) = 1$. We can check that the function
	\[
	\left(\frac{1-x}{e^{2t(x-1)}-x} \right)^{1/2} = \mathscr{F}(x,1/2; 2t) = \sum_{n=0}^\infty 2^n F_n (x, 1/2) \frac{t^n}{n!}
	\]
	satisfies these conditions by a direct calculation.
\end{proof}

The generating function $Q(x,t) = \mathscr{F}(x,1/2; 2t)$ tells us that the coefficients of $q_n(x)$ count perfect matchings with the restriction on the number of matching pairs have odd smaller entries (see~\cite{Ma2013} and \cite[A185411]{OEIS}) and $q_n(1)=(2n+1)!!$.

The relation between the polynomials $F_n(x,1/2)$ and $q_n(x)$  shed light on a proof of Stephan's conjecture which follows in the next section.

%-------------------------------------------------------------------
\section{The polynomials $p_n(x)$ and a proof of Stephan's conjecture} \label{s3}
%-------------------------------------------------------------------
In this section, we focus on the polynomial sequence $p_n(x)$ in the expression of Lehmer~\eqref{Lehmer-central-binomial}. We prove the observation of Stephan who noticed a relation of the sequence with the poly-Bernoulli numbers. Poly-Bernoulli numbers were introduced by Kaneko~\cite{Kaneko1997} by the polylogarithm function ($\mathrm{Li}_k(z) = \sum_{m=1}^\infty z^m/m^k$ for any integer $k$) as a generalization of the classical Bernoulli numbers. 
The \emph{poly-Bernoulli numbers} $B_n^{(k)} \in \Q$ are defined by
\[
	\sum_{n=0}^\infty B_n^{(k)} \frac{t^n}{n!} = \frac{\mathrm{Li}_k(1-e^{-t})}{1-e^{-t}}.
\]
Poly-Bernoulli numbers have attractive properties. In particular, the values with negative indices $k$ enumerate several combinatorial objects, (see for instance, \cite{BenyiHajnal2017, BenyiMatsusaka2021, Brewbaker2008, HMSY2021} and the references therein).

As one of the most basic properties, Arakawa and Kaneko~\cite{ArakawaKaneko1999} showed that
\[
	\sum_{k=0}^n (-1)^k B_{n-k}^{(-k)} = 0
\]
holds for any positive integer $n$. 
Since then, several authors have generalized the formula for the alternating anti-diagonal sum in~\cite{KanekoSakuraiTsumura2018, Matsusaka2020}, but not much is known about the anti-diagonal sum in~\cref{Stephan's conjecture}.

In most of the combinatorial interpretations, the roles of $n$ and $k$ are separately significant, hence it is not natural to consider the anti-diagonal sum. However, one of the interpretations, where this is natural, is the set of permutations with \emph{ascending-to-max property}~\cite{HeMunroRao2005}. A permutation $\pi \in \mathfrak{S}_n$ is called \emph{ascending-to-max}, if for any integer $i$, $1\leq i\leq n-2$
\begin{enumerate}
	\item[a.] if $\pi^{-1}(i)<\pi^{-1}(n)$ and $\pi^{-1}(i+1)<\pi^{-1}(n)$ then $\pi^{-1}(i)<\pi^{-1}(i+1)$, and
	\item[b.] if $\pi^{-1}(i)>\pi^{-1}(n)$ and  $\pi^{-1}(i+1)>\pi^{-1}(n)$, then  $\pi^{-1}(i)>\pi^{-1}(i+1)$. 
\end{enumerate}

In other words: record a permutation in one-line notation and draw an arrow from value $i$ to $i+1$ for each $i$. Then, the permutation has the ascending-to-max property if all the arrows starting from the left of $n$ point forward and all the arrows starting from an element to the right of $n$ point backward. For instance, $47518362$ has the property, but $41385762$ has not. 
It follows from the results of B\'{e}nyi and Hajnal \cite{BenyiHajnal2015} that the number of permutations $\pi \in \mathfrak{S}_{n+1}$ with the ascending-to-max property is given by the anti-diagonal sum $b_n = \sum_{k=0}^n B_{n-k}^{(-k)}$.
However, no explicit formula or recursion was known about the sequence $b_n$.

Our first result is a recursion for the sequence $b_n$.   
\begin{proposition}\label{bn-rec-lem}
	The sequence $(b_n)_{n \geq 0}$ satisfies the recurrence relation $b_0 = 1$ and 
	\begin{align}\label{cn-recur-b}
		3b_{n+1} = 2b_n + \sum_{k=0}^n {n+1 \choose k} b_k + 3.
	\end{align}
\end{proposition}

In order to prove this theorem, we need some preparations. 
 Recall that by~\cite[p.163]{ArakawaKaneko1999}, we have
\begin{align}\label{eq:generating-function-bn}
	\sum_{n=0}^\infty b_n x^n &= \sum_{j=0}^\infty \frac{(j!)^2 x^{2j}}{(1-x)^2 (1-2x)^2 \cdots (1-(j+1)x)^2} = \frac{1}{(1-x)^2} \sum_{j=0}^\infty f_j \left(2-\frac{1}{x}, 2-\frac{1}{x} \right),
\end{align}
where $(x)_j = x(x+1)(x+2) \cdots (x+j-1)$ is the Pochhammer symbol and we put
\[
f_j(x,y) = \frac{(j!)^2}{(x)_j (y)_j}.
\]
By a direct calculation, we have
\begin{align*}
	\sum_{n=0}^\infty \sum_{k=0}^n &{n+1 \choose k} b_k x^n = \sum_{k=0}^\infty b_k x^k \sum_{n=0}^\infty {n+k+1 \choose k} x^n = \sum_{k=0}^\infty b_k x^{k-1} \left(\frac{1}{(1-x)^{k+1}} - 1 \right)\\
	&= \frac{1-x}{x(1-2x)^2} \sum_{j=0}^\infty f_j \left(3-\frac{1}{x}, 3-\frac{1}{x} \right) - \frac{1}{x(1-x)^2} \sum_{j=0}^\infty f_j \left(2-\frac{1}{x}, 2- \frac{1}{x} \right).
\end{align*}
Thus, the desired recursion in~\eqref{cn-recur-b} is equivalent to
\begin{align}\label{key_equality}
	\frac{2(2-x)}{(1-x)^2} \sum_{j=0}^\infty f_j \left(2-\frac{1}{x}, 2-\frac{1}{x} \right) = \frac{3}{1-x} + \frac{1-x}{(1-2x)^2} \sum_{j=0}^\infty f_j \left(3-\frac{1}{x}, 3-\frac{1}{x} \right).
\end{align}
To prove \eqref{key_equality}, we derive a useful equation.
\begin{lemma}\label{3F2-trans}
	For any $j \in \Z_{\geq 0}$, we have
	\begin{align*}
		&(x-1)(x-2) \left(f_j(x-2,y) - f_{j-1}(x-2,y) \right) + (x-1)(2x-5) f_{j-1}(x-1, y)\\
		&-(x-1)(x-y-1)f_j(x-1,y) - (x-2)^2 f_{j-1}(x,y)\\
		&= \begin{cases}
			(x-1)(y-1) &\text{if } j= 0,\\
			0 &\text{if } j > 0,
		\end{cases}
	\end{align*}
	where we put $f_{-1}(x,y) = 0$.
\end{lemma}

\begin{proof}
	By direct calculation, one can verify it.
\end{proof}

\begin{proof}[Proof of \cref{bn-rec-lem}]
	We prove \eqref{key_equality}. Setting $x \rightarrow 3-1/x$ and $y\rightarrow 2-1/x$ and applying \cref{3F2-trans}, we obtain
	\begin{align*}
		f_j \left(2-\frac{1}{x}, 2-\frac{1}{x} \right) &= \frac{(1-x)^2}{(1-2x)(2-x)} f_j \left(3-\frac{1}{x}, 2-\frac{1}{x} \right)\\
		&\quad - \frac{1-x}{2-x} \left(f_{j+1} \left(1-\frac{1}{x}, 2-\frac{1}{x} \right) - f_j \left(1-\frac{1}{x}, 2- \frac{1}{x} \right) \right).
	\end{align*}
	Summing up both sides over $j = 0, 1, 2, \dots$, we have
	\begin{align*}
		\sum_{j=0}^\infty f_j \left(2-\frac{1}{x}, 2- \frac{1}{x} \right) = \frac{1-x}{2-x} + \frac{(1-x)^2}{(1-2x)(2-x)} \sum_{j=0}^\infty f_j \left(3-\frac{1}{x}, 2- \frac{1}{x} \right).
	\end{align*}
	From \cref{3F2-trans} again for $x \to 3-1/x$ and $y \to 3-1/x$, we conclude \eqref{key_equality}.
\end{proof}
\begin{remark}
	Unfortunately, we could not provide a combinatorial proof for this recurrence, though it would be very interesting to find one using for instance the permutations with the ascending-to-max property. 
\end{remark}

To relate the sequence $(b_n)_{n \geq 0}$ to the polynomial sequence $(p_n(x))_{n \geq -1}$, we next derive the generating function for $p_n(x)$.

\begin{proposition}\label{Pxt-generating-function}
	We have
	\begin{align*}
	P(x,t) &:= \sum_{n=0}^\infty p_{n-1}(x) \frac{t^n}{n!} = \frac{e^{(1-x)t} \left(\arcsin(x^{1/2} e^{(1-x)t}) - \arcsin(x^{1/2}) \right)}{x^{1/2} (1-x e^{2(1-x)t})^{1/2}}.
\end{align*}
\end{proposition}

\begin{proof}
	From \eqref{Lehmer-central-binomial}, we have
	\begin{align*}
		\sum_{k=0}^\infty &\sum_{n=1}^\infty \frac{(2n)^{k-1} (2x)^{2n}}{{2n \choose n}} \frac{t^k}{k!} = \frac{x}{(1-x^2)^{1/2}} \left(x \sqrt{1-x^2} P \left(x^2, \frac{t}{1-x^2} \right) + \arcsin(x) Q \left(x^2, \frac{t}{1-x^2} \right) \right).
	\end{align*}
	By applying \eqref{Lehmer-central-binomial} with $k=-1$ again, the left-hand side equals
	\begin{align*}
		\sum_{n=1}^\infty \frac{(2n)^{-1} (2x e^t)^{2n}}{{2n \choose n}} = \frac{x e^t \arcsin(x e^t)}{(1-x^2 e^{2t})^{1/2}}.
	\end{align*}
	Thus, Combining with \cref{Q-explicit}, we have
	\begin{align*}
		P\left(x^2, \frac{t}{1-x^2} \right) &= \frac{e^t \arcsin(x e^t)}{x(1-x^2 e^{2t})^{1/2}} - \frac{\arcsin(x)}{x\sqrt{1-x^2}} \mathscr{F} \left(x^2, \frac{1}{2}; \frac{2t}{1-x^2} \right)\\
		&= \frac{e^t \left(\arcsin(xe^t) - \arcsin(x) \right)}{x (1-x^2 e^{2t})^{1/2}},
	\end{align*}
	which implies the claim.
\end{proof}

We define the sequence $a_n$ as special values of $p_n(x)$,
\begin{align}\label{an-def}
	a_n = \left(\frac{2}{3} \right)^n p_n \left(\frac{1}{4} \right).
\end{align}

Using the generating function of $p_n(x)$, we obtain the recurrence formula that the sequence $(a_n)_{n \geq 0}$ satisfies.

\begin{proposition}\label{an-rec-lem}
	The sequence $(a_n)_{n \geq 0}$ defined in \eqref{an-def} satisfies $a_0 = 1$ and
	\begin{align*}
		3a_{n+1} = 2a_n + \sum_{k=0}^n {n+1 \choose k} a_k + 3.
	\end{align*}
\end{proposition}

\begin{proof}
	By \cref{Pxt-generating-function}, the generating function for $a_n$ is given by
	\begin{align}\label{eq:generating-function-an}
		\sum_{n=0}^\infty a_n \frac{t^{n+1}}{(n+1)!} = \frac{3}{2} P \left(\frac{1}{4}, \frac{2}{3}t \right) = \frac{6 e^{t/2} \left(\arcsin(e^{t/2}/2) - \arcsin(1/2) \right)}{(4- e^t)^{1/2}}.
	\end{align}
	Since this function satisfies the differential equation
	\[
		\left((4-e^t) \frac{d}{dt} - 2 \right) \frac{3}{2}P \left(\frac{1}{4}, \frac{2}{3}t \right) = 3e^t,
	\]
	the coefficients $a_n$ satisfy the desired recurrence formula.
\end{proof}

In conclusion, we have the main theorem.

\begin{theorem}\label{theo:Stephan}
	$\cref{Stephan's conjecture}$ is true, i.e.
	for any $n \geq 0$,
	\[
	\left(\frac{2}{3} \right)^n p_n \left(\frac{1}{4} \right) = \sum_{k=0}^n B_{n-k}^{(-k)}.
	\]
		
\end{theorem}
\begin{proof}
	\cref{bn-rec-lem} and \cref{an-rec-lem} imply the theorem.
\end{proof}
In the course of our proof, we obtain two types of generating functions in \eqref{eq:generating-function-bn} and \eqref{eq:generating-function-an} for the sequences $(a_n)_{n \geq 0} = (b_n)_{n \geq 0}$.
As a corollary, we have an explicit formula for the anti-diagonal sum, (see~\cite[p.24]{BenyiHajnal2017}).
\begin{corollary}
	\[
	b_n = \sum_{k=0}^{n}B_{n-k}^{(-k)}= \frac{(-1)^{n+1}}{2}\sum_{j=1}^{n+1}(-1)^j
	j!\sts{n+1}{j}\frac{\binom{2j}{j}}{3^{j-1}} \sum_{i=0}^{j-1}\frac{3^i}{(2i+1)
		\binom{2i}{i}}.
	\]
\end{corollary}
\begin{proof}
	The result follows from the explicit formula by Borwein--Girgensohn~\cite{BorweinGirgensohn2005} and \cref{theo:Stephan}.
\end{proof}

As a final remark, we show that the polynomial $p_n(x)$ can also be expressed in terms of bivariate Eulerian polynomials.

\begin{theorem}\label{explicit-P}
	For any $n \geq 0$, we have
	\[
		p_n(x) = 2^n \sum_{k=0}^n {n+1 \choose k} F_{n-k}(x,1/2) F_k(x,1/2).
	\]
\end{theorem}

\begin{proof}
	Consider
	\begin{align*}
		P \left(x, t \right) 
	   	&= \frac{e^{(1-x)t} \left(\arcsin(x^{1/2} e^{(1-x)t}) - \arcsin(x^{1/2}) \right)}{x^{1/2} (1-x e^{2(1-x)t})^{1/2}}\\
		&= \mathscr{F}\left(x,\frac{1}{2}; 2t \right) \frac{1}{x^{1/2} (1-x)^{1/2}} \left(\arcsin(x^{1/2} e^{(1-x)t}) - \arcsin(x^{1/2}) \right).
	\end{align*}
	Since
	\begin{align*}
		\frac{d}{dt} \frac{1}{x^{1/2}(1-x)^{1/2}} \bigg(\arcsin(x^{1/2} e^{(1-x)t}) - \arcsin(x^{1/2}) \bigg) &= \mathscr{F}\left(x, \frac{1}{2}; 2t \right),
	\end{align*}
	it holds that
	\[
		\frac{1}{x^{1/2}(1-x)^{1/2}} \left(\arcsin(x^{1/2} e^{(1-x)t}) - \arcsin(x^{1/2}) \right) = \sum_{n=0}^\infty 2^n F_n(x, 1/2) \frac{t^{n+1}}{(n+1)!}.
	\]
	Thus, we have
	\begin{align*}
		P (x, t) = \sum_{n=1}^\infty \left(2^{n-1}\sum_{k=0}^{n-1} {n \choose k} F_{n-k-1}(x,1/2) F_k (x,1/2) \right) \frac{t^n}{n!},
	\end{align*}
	which concludes the proof.
\end{proof}

\bibliographystyle{amsplain}%{amsalpha}
\bibliography{References1} 

\end{document}